\documentclass[12pt]{article}
\usepackage{authblk}
\usepackage[a4paper,margin=1.23in]{geometry}
\usepackage{amsmath,amssymb,amsthm}
\usepackage{mathrsfs}
\usepackage{cite}
\usepackage[colorlinks=true,citecolor=blue,linkcolor=blue,urlcolor=blue]{hyperref}

\newtheorem{theorem}{Theorem}
\newtheorem{lemma}{Lemma}

\newtheorem{conjecture}{Conjecture}
\theoremstyle{definition}
\newtheorem{definition}{Definition}
\theoremstyle{remark}
\newtheorem{remark}{Remark}

\DeclareMathOperator{\Tr}{tr}
\newcommand{\eps}{\varepsilon}

\allowdisplaybreaks

\title{Proofs of two conjectures on generalizations of Brouwer's Laplacian conjecture} 
\author{Junying Lu\thanks{ School of Mathematics and Statistics, Nanjing University of Science and Technology, Nanjing 210094, P.R. China (Email: lujunying\_math@163.com)}~, Jiabao Yang\thanks{School of Mathematics, Nanjing University, Nanjing 210093, P.R. China (Email: jbyang1215@nju.edu.cn)}}

\date{}

\begin{document}
\maketitle

\begin{abstract}
Let $G=(V,E)$ be a simple graph of order $n$ and let
$\lambda_1(G)\ge \cdots \ge \lambda_n(G)$ be the eigenvalues of its Laplacian matrix. Brouwer conjectured that for every $1\le k\le n$, $\sum_{i=1}^k\lambda_i(G)\le |E|+\binom{k+1}{2}$, which was recently confirmed by Kothari and Tudose. Before Brouwer’s conjecture was proved, Lew (JCT-B, 2026) established a weaker form of Brouwer’s Laplacian eigenvalue inequality and proposed two conjectures for upper bounds on the sum of the $k$ largest Laplacian eigenvalues, one in terms of the matching number and the other in terms of the vertex-cover number.  Using Brouwer's Laplacian inequality, we prove both conjectures.
\end{abstract}

\section{Introduction}
All graphs in this note are finite and simple. Let $G=(V(G),E(G))$ be a graph with $|V(G)|=n$, and write
$e(G)=|E(G)|$. For a vertex $v\in V(G)$, the degree of $v$ in $G$, denoted by $d(v)$, is the number of edges incident to $v$. The Laplacian matrix of $G$, denoted by $L(G)$, is the $n\times n$ matrix defined by 
\[
L(G)_{u,v}=\begin{cases}
    d(u), & \text{if } u=v; \\
    -1, & \text{if } uv\in E(G); \\
    0, & \text{otherwise}.
\end{cases}   
\]
for all $u,v\in V(G)$. We order the eigenvalues of
$L(G)$ as
\[
 \lambda_1(G)\ge \lambda_2(G)\ge \cdots \ge \lambda_n(G)=0.
\]
For $0\le k\le n$, set
\[
 s_k(G)=\sum_{i=1}^k \lambda_i(G).
\]

Brouwer (see Section 3.11.1 in \cite{brouwer2012}) proposed the following conjecture about the sum of the $k$ largest Laplacian eigenvalues of a graph $G$.
\begin{conjecture}[Brouwer's Laplacian conjecture]\label{conj-B}
For every graph of order $n$ and every integer $1\le k\le n$,
\[
 s_k(G)\le e(G)+\binom{k+1}{2}.
\]    
\end{conjecture}

Conjecture \ref{conj-B} has motivated a substantial body of work on extremal problems for partial sums of Laplacian eigenvalues. The boundary cases \(k=1\), \(k=n-1\), and \(k=n\) follow from standard spectral estimates or from the trace identity, while the first nontrivial case \(k=2\) was proved by Haemers, Mohammadian and Tayfeh-Rezaie \cite{HaemersMohammadianTayfehRezaie2010}. Subsequently, the case \(k=3\) was settled by Wang, Lin, Zhang and Ye \cite{WangLinZhangYe2026}. 
 
Beyond results for fixed values of $k$, Conjecture \ref{conj-B} has been verified for numerous graph classes, including trees, unicyclic and bicyclic graphs, threshold graphs, split graphs, cographs, and regular graphs; see, e.g., \cite{DuZhou2012,Mayank2010,Berndsen2012}. It is also known to hold for graph families subject to restrictions involving girth, clique number, vertex-cover number, diameter, the number of pendant vertices, arboricity, or degree variance; see, e.g., \cite{Chen2018,Chen2019,GaniePirzadaRatherTrevisan2020,Cooper2021}. 
Further developments include preservation results under graph operations~\cite{LinWang2025} and reductions via a graph-dependent Brouwer critical index~\cite{TorresTrevisan2026}.

Another line of work concerns approximate and parameter-dependent bounds. Lew established the universal estimates $s_k(G)\le e(G)+k^2+15k\log k+65k$~\cite{Lew2026b} and $s_k(G)\le e(G)+k^2$~\cite{Lew2025}, as well as the asymptotically sharper estimate $s_k(G)\le e(G)+\binom{k}{2}+(4k-2)\sqrt{k}$~\cite{Lew2026a}. He also proved the matching number bound $s_k(G)\le e(G)+k\nu(G)+\lfloor k/2\rfloor$~\cite{Lew2026b}, where $\nu(G)$ denotes the matching number of $G$. In the same paper, Lew proposed the following two conjectures.
\begin{conjecture}[Lew \cite{Lew2026b}]\label{conj:matching}
Let $G$ be a graph with $n$ non-isolated vertices.  If $1\le k\le n-2$, then
\[
 s_k(G)\le e(G)+k\nu(G).
\]
\end{conjecture}

Let $\tau(G)$ be the vertex-cover number of $G$.
\begin{conjecture}[Lew \cite{Lew2026b}]\label{conj:cover}
Let $G$ be a graph of order $n$.  If $\tau(G)\le k\le n$, then
\[
 s_k(G)\le e(G)+k\tau(G)-\binom{\tau(G)}{2}.
\]
\end{conjecture}
The bound in Conjecture \ref{conj:matching} is stronger than that in Conjecture \ref{conj-B} for $2\nu(G) \le k\le n$, and the bound in Conjecture \ref{conj:cover} is stronger than that in Conjecture \ref{conj-B} for $\tau(G)<k\le n$.

More recently, Kothari and Tudose \cite{KothariTudose2026} proved Conjecture \ref{conj-B} by deriving it from the Grone--Merris--Bai theorem \cite{GroneMerris1994,Bai} for split graphs. They also proved the converse implication, thereby establishing the equivalence between Conjecture \ref{conj-B} and the Grone--Merris--Bai theorem. We shall use their result in the following form.

\begin{theorem}[Kothari and Tudose \cite{KothariTudose2026}]
\label{thm:brouwer}
Conjecture~\ref{conj-B} holds.
\end{theorem}

Li and Guo \cite{LiGuo2022} conjectured the corresponding characterization of equality. Cai, Chen, Yang, and Zhang \cite{CaiChenYangZhang2026} subsequently proved that, for each $1\le k\le n-1$, equality in Brouwer's bound holds if and only if $G$ is a threshold graph with clique number $k+1$.

In this paper, we prove both conjectures proposed by Lew.

\begin{theorem}\label{thm:main}
Conjectures~\ref{conj:matching} and \ref{conj:cover} hold.
\end{theorem}

The proofs of Conjectures~\ref{conj:matching} and \ref{conj:cover}  will be presented in Sections \ref{sec:matching} and \ref{sec:cover}, respectively.
 
\section{Preliminaries}
We denote $\eps_k(G)=s_k(G)-e(G)$ for $0\le k\le n$, where $s_0(G)=0$.
We first present some upper bounds of  $\eps_k(G)$.

\begin{lemma}[Haemers, Mohammadian, and Tayfeh-Rezaie~\cite{HaemersMohammadianTayfehRezaie2010}]\label{lem:subadditive}
Let $G_1,\ldots,G_t$ be edge-disjoint graphs and let $G$ be their union.  Then,
for every $1\le k\le |V(G)|$,
\[
 \eps_k(G)\le \sum_{j=1}^t \eps_k(G_j),
\]
where isolated vertices may be added to the graphs so that the Laplacian
matrices have the same order.
\end{lemma}


\begin{lemma}[Das, Mojallal, and Gutman \cite{Das2015}]\label{lem:coverbasic}
For every graph $G$ and every $1\le k\le |V(G)|$,
\[
 \eps_k(G)\le k\tau(G).
\]
\end{lemma}

\begin{lemma}\label{lem:trace}
For every graph $G$ and every $1\le k\le |V(G)|$,
\[
 \eps_k(G)\le e(G).
\]
\end{lemma}

\begin{proof}
The Laplacian matrix has trace equal to the sum of all vertex degrees, hence
\[
 \sum_{i=1}^{|V(G)|}\lambda_i(G)=\Tr L(G)=2e(G).
\]
Since all Laplacian eigenvalues are non-negative, every partial sum is at most
the total sum.  Thus $s_k(G)\le 2e(G)$, and consequently
$\eps_k(G)=s_k(G)-e(G)\le e(G)$.
\end{proof}

We also use Edmonds' characterization of the matching number.

\begin{definition}
An \emph{odd set cover} of a graph $G=(V,E)$ is a family
\[
 \mathcal C=\{v_1,\ldots,v_r,S_1,\ldots,S_t\},
\]
where $v_1,\ldots,v_r$ are distinct vertices and $S_1,\ldots,S_t$ are pairwise disjoint
subsets of $V$ of odd cardinality, such that every edge of $G$ is either incident
with one of the vertices $v_j$, or is contained in one of the sets $S_i$.  Its
weight is
\[
 w(\mathcal C)=r+\sum_{i=1}^t \frac{|S_i|-1}{2}.
\]
\end{definition}

\begin{theorem}[Edmonds~\cite{edmonds}]\label{thm:edmonds}
For every graph $G$,
\[
 \nu(G)=\min\{w(\mathcal C):\mathcal C \text{ is an odd set cover of }G\}.
\]
\end{theorem}

The next elementary complement lemma will be used for the vertex-cover inequality.
It was proved by Brouwer and Haemers~\cite{brouwer2012}.
For the sake of completeness, we include a proof here.

\begin{lemma}[Brouwer and Haemers~\cite{brouwer2012}] \label{lem:complement}
Let $G$ be a graph on $n$ vertices and let $\overline G$ be its complement.  Then
\[
 \lambda_i(\overline G)=n-\lambda_{n-i}(G)
 \qquad (1\le i\le n-1).
\]
Consequently, if $0\le k\le n-1$ and $q=n-k-1$, then
\[
 s_k(G)=2e(G)-qn+s_q(\overline G).
\]
\end{lemma}

\begin{proof}
The Laplacian of the complete graph $K_n$ is $L(K_n)=nI-J$, where $J$ is the all-ones matrix.  Since $G$ and $\overline G$ partition the edge set of $K_n$,
\[
 L(\overline G)=L(K_n)-L(G)=nI-J-L(G).
\]
Both $L(G)$ and $L(\overline G)$ annihilate the all-ones vector ${\bf 1}$.  On the orthogonal complement ${\bf 1}^{\perp}$ the matrix $J$ is zero.  Hence, if
$x\in {\bf 1}^{\perp}$ and $L(G)x=\lambda x$, then
\[
 L(\overline G)x=(nI-L(G))x=(n-\lambda)x.
\]
The eigenvectors corresponding to the non-zero part of the Laplacian spectrum of $G$, together with any additional zero eigenvectors orthogonal to ${\bf 1}$, lies in ${\bf 1}^{\perp}$.  Therefore
the eigenvalues of $L(\overline G)$ are
\[
 n-\lambda_{n-1}(G),\ldots,n-\lambda_1(G),0,
\]
listed in non-increasing order as stated.

Now let $q=n-k-1$.  If $q=0$, then $s_q(\overline G)=0$ and the formula says
$s_{n-1}(G)=2e(G)$, which is true because $\lambda_n(G)=0$.  If $q\ge 1$, then
\[
 s_q(\overline G)=\sum_{i=1}^q \bigl(n-\lambda_{n-i}(G)\bigr)
 =qn-\sum_{j=k+1}^{n-1}\lambda_j(G).
\]
Since $\sum_{j=1}^{n-1}\lambda_j(G)=2e(G)$, rearranging gives
$s_k(G)=2e(G)-qn+s_q(\overline G)$.
\end{proof}

\begin{lemma}\label{lem:clique}
Let $G$ be a graph containing a clique of order $a\ge 1$.  Then, for every
integer $0\le q\le a-1$,
\[
 s_q(G)\le e(G)+qa-\binom{a}{2}.
\]
\end{lemma}

\begin{proof}
The case $q=0$ is immediate, because $G$ contains at least the
$\binom a2$ edges of the clique and hence
$0\le e(G)-\binom a2$.
Assume $1\le q\le a-1$.

Let $H$ be the graph obtained from this $a$-clique by adding all remaining
vertices of $G$ as isolated vertices.  Then
$L(G)-L(H)$ is a sum of edge Laplacians and is positive semidefinite.  By Weyl's
monotonicity theorem, $\lambda_i(G)\ge \lambda_i(H)$ for all $i$.
The non-zero Laplacian eigenvalues of $H$ are equal to $a$, with multiplicity
$a-1$.  Therefore $\lambda_i(G)\ge a$ for $1\le i\le a-1$.
Applying Theorem~\ref{thm:brouwer} to $G$ with $a-1$ in place of $k$ gives $s_{a-1}(G)\le e(G)+\binom a2$.
Since every eigenvalue $\lambda_{q+1}(G),\ldots,\lambda_{a-1}(G)$ is at
least $a$, we obtain
\begin{align*}
 s_q(G)
 &=s_{a-1}(G)-\sum_{i=q+1}^{a-1}\lambda_i(G)  \\
 &\le e(G)+\binom a2-a(a-1-q) \\
 &=e(G)+qa-\binom a2.
\end{align*}
This proves the lemma.
\end{proof}

\section{Proof of Conjecture~\ref{conj:matching}}\label{sec:matching}
\begin{theorem}\label{thm:matching}
Let $G$ be a graph with $n$ non-isolated vertices.  If $1\le k\le n-2$, then
\[
 s_k(G)\le e(G)+k\nu(G).
\]
\end{theorem}

\begin{proof}
Deleting isolated vertices does not change $e(G)$, $\nu(G)$, or $s_k(G)$.  Thus, we may assume that $G$ has no isolated vertices
and that $|V(G)|=n$.

For convenience, we write $\nu=\nu(G)$.  We split the proof into two cases.

\smallskip
\noindent{\bfseries Case 1. $k\le 2\nu-1$.}

By Theorem~\ref{thm:brouwer},
\[
 \eps_k(G)=s_k(G)-e(G)\le \binom{k+1}{2}=\frac{k(k+1)}{2}.
\]
The assumption $k\le 2\nu-1$ is equivalent to $(k+1)/2\le \nu$.  Hence
\[
 \binom{k+1}{2}=k\frac{k+1}{2}\le k\nu,
\]
and therefore $s_k(G)\le e(G)+k\nu$.

\smallskip
\noindent{\bfseries Case 2. $k\ge 2\nu$.}

Since $k\le n-2$, we have $n\ge k+2\ge 2\nu+2$.

By Theorem~\ref{thm:edmonds}, choose an odd set cover
$\mathcal C=\{v_1,\ldots,v_r,S_1,\ldots,S_t\}$ with weight $w(\mathcal C)=\nu$.  
Let
\[
 |S_i|=2a_i+1 \quad (a_i\ge 0)
 \quad \text{and} \quad
 \mu=\sum_{i=1}^t a_i.
\]
Clearly, we have $r+\mu=\nu$.
Let
\[
 R=\{v_1,\ldots,v_r\},\qquad
 E_1=\{e\in E(G):e\cap R\ne\emptyset\},
\]
and define
\[
 G_1=(V(G),E_1), \qquad G_2=(V(G),E(G)\setminus E_1).
\]
The graphs $G_1$ and $G_2$ are edge-disjoint, and their union is $G$.

Note that the set $R$ is a vertex cover of $G_1$.  
By Lemma~\ref{lem:coverbasic}, $\eps_k(G_1)\le kr$.
By the definition of $\mathcal{C}$, every edge of $G_2$ is contained in one of the odd sets $S_i$.  Since the sets
$S_i$ are pairwise disjoint, we have 
\begin{equation}\label{3.4}
e(G_2)\le \sum_{i=1}^t \binom{2a_i+1}{2}=\sum_{i=1}^t a_i(2a_i+1).   
\end{equation}
By Lemma~\ref{lem:trace}, $\eps_k(G_2)\le e(G_2)$.
It remains to prove
\begin{equation}\label{3.6}
e(G_2)\le k\mu. 
\end{equation}

First, suppose that $r\ge 1$.  Then, by $r+\mu=\nu$ and the present case assumption,
$k\ge 2\nu=2(r+\mu)\ge 2\mu+2$.
Moreover $\sum_i a_i^2\le (\sum_i a_i)^2=\mu^2$.  Using \eqref{3.4},
\[
 e(G_2)
 \le 2\sum_{i=1}^t a_i^2+
\sum_{i=1}^t a_i
 \le 2\mu^2+
\mu
 \le (2\mu+2)\mu
 \le k\mu.
\]
Thus \eqref{3.6} holds when $r\ge 1$.

Now suppose that $r=0$.  Then $\mu=\nu$ and $k\ge 2\mu$.  
Then $n\ge 2\nu+2=2\mu+2$.
We claim that at least two of the integers $a_i$ are positive.  Indeed, if exactly
one of them were positive, say $a_j=\mu$, then every edge of $G$ would be
contained in $S_j$, because $r=0$.  Since $G$ has no isolated vertices, every
vertex of $G$ would have to lie in $S_j$.  Hence $n\le |S_j|=2\mu+1$,
which is a contradiction to $n\ge 2\mu+2$.  Therefore, at least two $a_i$ are positive, and necessarily $\mu\ge 2$.  For non-negative
integers $a_i$ with sum $\mu$ and at least two positive terms, the sum of squares
is maximized by the distribution $\mu-1,1,0,\ldots,0$.  Hence
\[
 \sum_{i=1}^t a_i^2\le (\mu-1)^2+1.
\]
Therefore, it follows from \eqref{3.4} that
\begin{align*}
 e(G_2)\le 2\sum_{i=1}^t a_i^2+\mu \le 2\bigl((\mu-1)^2+1\bigr)+\mu =2\mu^2-3\mu+4 \le 2\mu^2 \le k\mu.
\end{align*}
Thus \eqref{3.6} holds also when $r=0$.

Combining Lemma~\ref{lem:subadditive}, $\eps_k(G_1)\le kr$, and $\eps_k(G_2)\le e(G_2)\le k\mu$, we get
\[
 \eps_k(G)
 \le \eps_k(G_1)+\eps_k(G_2)
 \le kr+k\mu
 =k(r+\mu)
 =k\nu.
\]
Hence $s_k(G)\le e(G)+k\nu(G)$.
\end{proof}

\begin{remark}
The bound in Theorem \ref{thm:matching} is sharp. Indeed, the Laplacian spectrum of $K_{1,n-1}$ consists of $n$, $1$ with multiplicity $n-2$, and $0$. Therefore,
$s_k(K_{1,n-1})=e(K_{1,n-1})+k\nu(K_{1,n-1})$
for every $1\le k\le n-2$.
\end{remark}

\section{Proof of Conjecture~\ref{conj:cover}}\label{sec:cover}
\begin{theorem}\label{thm:cover}
Let $G=(V,E)$ be a graph on $n$ vertices and put $t=\tau(G)$.  If
$t\le k\le n-1$, then
\[
 s_k(G)\le e(G)+kt-\binom t2.
\]
Moreover, $s_n(G)\le e(G)+nt-\binom{t+1}{2}$. In particular, Conjecture \ref{conj:cover} holds.
\end{theorem}

\begin{proof}
If $k=0$, then $t=0$ and the inequality says $0\le e(G)$, so assume $k\ge 1$.
Let $m=e(G)$ and let $C$ be a minimum vertex cover.
Then $|C|=t$ and $I=V\setminus C$ is an independent set of $G$.  Hence $I$ is a clique in the complement $\overline G$.  Put
\[
 a=|I|=n-t
 \quad \text{and} \quad
 \overline m=e(\overline G)=\binom n2-m.
\]

First, we assume $k\le n-1$.  Let $q=n-k-1$.
Since $k\ge t$, we have $q=n-k-1\le n-t-1=a-1$.
Also $q\ge 0$ because $k\le n-1$.  Applying Lemma~\ref{lem:clique} to the graph
$H=\overline G$, which contains the clique $I$ of order $a$, gives
\begin{equation}\label{4.1}
s_q(\overline G)\le \overline m+qa-\binom a2.  
\end{equation}
By Lemma~\ref{lem:complement},
\begin{equation}\label{4.2}
s_k(G)=2m-qn+s_q(\overline G).    
\end{equation}

Substituting \eqref{4.1} into \eqref{4.2}, and using $a=n-t$, yields
\begin{align}\label{4.3}
 s_k(G)
 &\le 2m-qn+\overline m+q(n-t)-\binom{n-t}{2} \nonumber\\
 &=2m-qn+\left(\binom n2-m\right)+q(n-t)-\binom{n-t}{2} \nonumber\\
 &=m+\binom n2-qt-\binom{n-t}{2}. 
\end{align}
Since $q=n-k-1$,  we conclude
\begin{align}\label{4.4}
 \binom n2-qt-\binom{n-t}{2}
 &=\binom n2-(n-k-1)t-\binom{n-t}{2} \nonumber\\
 &=kt-\binom t2. 
\end{align}
Combining \eqref{4.3} and \eqref{4.4} proves $s_k(G)\le m+kt-\binom t2$ for $t\le k\le n-1$.

It remains to consider $k=n$.  Since all Laplacian eigenvalues sum to the trace,
\begin{equation}\label{4.5}
s_n(G)=\Tr L(G)=2m.    
\end{equation}
Because $C$ is a vertex cover, every edge of $G$ either has both endpoints in
$C$ or has one endpoint in $C$ and the other in $V\setminus C$.  Hence
\begin{equation}\label{4.6}
m\le \binom t2+t(n-t)=nt-\binom{t+1}{2}.    
\end{equation}
Equations \eqref{4.5} and \eqref{4.6} give
\[
 s_n(G)=2m\le m+nt-\binom{t+1}{2},
\]
which is the desired inequality for $k=n$.
\end{proof}
\begin{remark}
The bounds in Theorem \ref{thm:cover} are sharp. For
$1\le t\le n-1$, let
$S_{n,t}$ be the graph obtained from a clique of order $t$ and an independent set of order $n-t$ by adding all edges between the two parts. Then
$\tau(S_{n,t})=t$, and the Laplacian spectrum of $S_{n,t}$
consists of $n$ with multiplicity $t$, $t$ with multiplicity
$n-t-1$, and $0$. One can verify that the equality in Theorem \ref{thm:cover} holds for every $t\le k\le n$.
\end{remark}

\section*{Declarations}
The authors declare that they have no known competing financial interests or personal relationships that could have appeared to influence the work reported in this paper.
\section*{\bf\Large Availability of Data and Materials}  
Not applicable.

\section*{Acknowledgements}
Conjecture~\ref{conj:matching} was  proved independently and concurrently by Huang and Qin, see \cite{HuangQin2026}. 
This research was supported  by National Key R\&D Program of China under grant number 2024YFA1013900 and NSFC under grant number 12471327.

\end{document}